\documentclass[11pt]{article}

\usepackage{enumerate}
\usepackage{amsmath}
\usepackage{amssymb,latexsym}
\usepackage{amsthm}

\usepackage{anysize}
\usepackage{graphicx}
\usepackage{url}
\usepackage{color}

\newcommand{\cH}{{\mathcal H}}

\newcommand{\cX}{{\mathcal X}}
\newcommand{\cC}{{\mathcal C}}
\newcommand{\cO}{{\mathcal O}}
\newcommand{\cS}{{\mathcal S}}
\newcommand{\F}{{\mathbb F}}

\newcommand{\fqs}{{\mathbb F_{q^2}}}
\newcommand{\aut}{{\rm Aut}}
\newcommand{\PGU}{{\rm PGU}}
\newcommand{\PSU}{{\rm PSU}}
\newcommand{\PGL}{{\rm PGL}}
\newcommand{\PSL}{{\rm PSL}}
\newcommand{\PG}{{\rm PG}}

\newcommand{\la}{\langle}
\newcommand{\ra}{\rangle}

\renewcommand{\mod}{\hbox{{\rm mod}\,}}

\newtheorem{theorem}{Theorem}[section]

\newtheorem{lemma}[theorem]{Lemma}
\newtheorem{corollary}[theorem]{Corollary}

\newtheorem{proposition}[theorem]{Proposition}

\begin{document}

\title{On maximal curves that are not quotients of the Hermitian curve}
\author{Massimo Giulietti, Maria Montanucci, and Giovanni Zini}

\maketitle
For each prime power $\ell$ the plane curve $\mathcal X_\ell$ with equation $Y^{\ell^2-\ell+1}=X^{\ell^2}-X$ is maximal over $\F_{\ell^6}$. Garcia and Stichtenoth in 2006 proved that $\mathcal X_3$ is not Galois covered by the Hermitian curve and raised the same question for $\mathcal X_\ell$ with $\ell>3$; in this paper we show that $\mathcal X_\ell$ is not Galois covered by the Hermitian curve for any $\ell>3$. Analogously, Duursma and Mak proved that the generalized GK curve $\mathcal C_{\ell^n}$ over $\F_{\ell^{2n}}$ is not a quotient of the Hermitian curve for $\ell>2$ and $n\ge 5$, leaving  the case $\ell=2$ open; here we show that $\mathcal C_{2^n}$ is not Galois covered by the Hermitian curve over $\F_{2^{2n}}$ for $n\geq5$. 
\section{Introduction}

Let $\fqs$ be the finite field with $q^2$ elements, where $q$ is a power of a prime $p$, and let $\cX$ be an $\fqs$-rational curve, that is a projective, absolutely irreducible, non-singular algebraic curve defined over $\fqs$. $\cX$ is called $\fqs$-maximal if the number $\cX(\fqs)$ of its $\fqs$-rational points attains the Hasse-Weil upper bound
$$ q^2+1+2gq, $$
where $g$ is the genus of $\cX$. Maximal curves have interesting properties and have also been investigated for their applications in Coding Theory. Surveys on maximal curves are found in \cite{FT,G,G2,GS,V,V2} and \cite[Chapt. 10]{HKT}.

The most important example of an $\fqs$-maximal curve is the Hermitian curve $\cH_q$, defined as any $\fqs$-rational curve projectively equivalent to the plane curve with Fermat equation
$$ X^{q+1}+Y^{q+1}+T^{q+1}=0. $$
The norm-trace equation
$$ Y^{q+1}=X^qT+XT^q $$
gives another model of $\cH_q$, $\fqs$-equivalent to the Fermat model, see \cite[Eq. (2.15)]{GSX}.
For fixed $q$, $\cH_q$ has the largest possible genus $g(\cH_q)=q(q-1)/2$ that an $\fqs$-maximal curve can have.
The automorphism group $\aut(\cH_q)$ is isomorphic to $\PGU(3,q)$, the group of projectivities of $\PG(2,q^2)$ commuting with the unitary polarity associated with $\cH_q$. %$\PGU(3,q)$ has a normal subgroup $\PSU(3,q)$ of index $\gcd(3,q+1)$ consisting of the %unitary transformations whose determinant is a cube in $\fqs$.

By a result commonly attributed to Serre, see \cite[Prop. 6]{L}, any $\fqs$-rational curve which is $\fqs$-covered by an $\fqs$-maximal curve is also $\fqs$-maximal. In particular, $\fqs$-maximal curves are given by the Galois $\fqs$-subcovers of an $\fqs$-maximal curve $\cX$, that is by the quotient curves $\cX/G$ over a finite $\fqs$-automorphism group $G\leq\aut(\cX)$.

Most of the known maximal curves are Galois subcovers of the Hermitian curve, many of which were studied in \cite{CKT,CKT2,GSX}. Garcia and Stichtenoth \cite{GS2} discovered the first example of maximal curve not Galois covered by the Hermitian curve, namely the curve $Y^7=X^9-X$ maximal over $\F_{3^6}$. It is a special case of the curve $\mathcal X_\ell$ with equation
\begin{equation}\label{abq}
 Y^{\ell^2-\ell+1}=X^{\ell^2}-X, 
\end{equation}
which is $\F_{\ell^{6}}$-maximal for any $\ell \ge 2$.
% in \cite{ABQ} by Abd\'on, Bezerra, and Quoos, for any $q\geq2$ and $n\geq3$ odd.
In \cite{GK}, Giulietti and Korchm\'aros showed that the Galois covering of $\mathcal X_\ell $ given by 
$$ \begin{cases} Z^{\ell^2-\ell+1}=Y^{\ell^2}-Y \\ Y^{\ell+1}=X^\ell+X \end{cases} $$
is also $\F_{\ell^6}$-maximal, for any prime power $\ell$. Remarkably, it is not covered by $\cH_{\ell^3}$ for any $\ell>2$. This curve, nowadays referred to as the GK curve, was generalized in \cite{GGS} by Garcia, G\"uneri, and Stichtenoth to the curve
$$ \cC_{\ell^n}:  \begin{cases} Z^\frac{\ell^n+1}{\ell+1}=Y^{\ell^2}-Y \\ X^\ell+X=Y^{\ell+1} \end{cases}, $$
which is $\F_{\ell^{2n}}$-maximal for any prime power $\ell$ and $n\geq3$ odd. For $\ell=2$ and $n=3$, $\cC_8$ is Galois covered by $\cH_8$, see \cite{GK}. Duursma and Mak proved in \cite{DM} that, if $\ell\geq3$, then $\cC_{\ell^n}$ is not Galois covered by $\cH_{\ell^n}$.
In Section $3$, we show that the same holds in the remaining open cases.

\begin{theorem}\label{result1}
For $\ell=2$ and $n\geq5$, $\cC_{2^n}$ is not a Galois subcover of the Hermitian curve $\cH_{\ell^n}$.
\end{theorem}

Duursma and Mak \cite[Thm. 1.2]{DM} showed that
if $\cC_{2^n}$ is the quotient curve $\cH_{2^n}/G$ for  $G$ a subgroup of $\aut(\cH_{2^n})$, then $G$ has order $(2^n+1)/3$ and acts semiregularly on $\cH_{2^n}$.
We investigate all subgroups $G$ of $\aut(\cH_{2^n})$ satisfying these conditions, relying also on classical results by Mitchell \cite{M} and Hartley \cite{H} (see Section $2$) which provide a classification of the maximal subgroups of $\PSU(3,q)$ in terms of their order and their action on $\cH_q$. For any candidate subgroup $G$, we find another subgroup $\bar G$ of  $\aut(\cH_{2^n})$ containing $G$ as a normal subgroup, and such that $\bar G/G$ has an action on $\cH_{2^n}/G$ not compatible with the action of any automorphism group of $\cC_{2^n}$.

In Section $4$ we consider the curve $\mathcal X_\ell$ with equation \eqref{abq}. In \cite{GS2} it was shown that $\cX_3$ is not a Galois subcover of $\cH_{3^6}$ by \cite{GS2}, while $\cX_2$ is a quotient of $\cH_{2^6}$, as noted in \cite{GT}. Garcia and Stichtenoth \cite[Remark 4]{GS2} raised the same question for any $\ell>3$. The case where $\ell$ is a prime was settled by  Mak \cite{Mak}. Here we provide an answer for any prime power $\ell>3$.

\begin{theorem}\label{result2}
For $\ell>3$, $\cX_\ell$ is not a Galois subcover of the Hermitian curve $\cH_{\ell^6}$.
\end{theorem}

In the proof of Theorem \ref{result2} we bound the possible degree of a Galois covering $\cH_{\ell^6}\rightarrow\cX_\ell$ by means of \cite[Thm. 1.3]{DM}, then we exclude the three possible values given by the bound. To this aim, we use again the classification results of Mitchell \cite{M} and Hartley \cite{H}, other group-theoretic arguments, and the Riemann-Hurwitz formula 
%and the Hilbert different formula
 (see \cite[Chapt. 3]{Sti}) applied to the Galois coverings $\cH_{\ell^6}\rightarrow\cH_{\ell^6}/G$.

\section{Preliminary results}
%Let $\cH_q$ be the Hermitian curve, defined over $\F_q$ and maximal over $\F_{q^2}$, with affine equation
%$$ \cH_q:\quad X^{q}+X=Y^{q+1}. $$
%The automorphism group of the Hermitian curve $\cH_q$ is $\aut(\cH_q)\cong\PGU(3,q)$, the group of projective transformations of $\PG(2,q^2)$ commuting with the unitary polarity associated with $\cH_q$.
%Hence any Galois subcover of $\cH_q$ is $\fqs$-covered by $\cH_q$.
%
%$\PGU(3,q)$ has a normal subgroup $\PSU(3,q)$ of index $\gcd(q+1,3)$, made of the unitary transformations whose determinant is a cube in $\F_{q^2}$.
%Maximal subgroups of $\PSU(3,q)$ are classified up to conjugacy. This result is due to Mitchell \cite[Par. 16, Page 241]{M} and Hartley \cite[Page 158]{H}, for $q$ odd and $q$ even respectively; see also \cite[Thm. A.10]{HKT}. %We will use this classification only for $q$ even.

\begin{theorem}\label{MH}
{\rm (Mitchell \cite{M}, Hartley \cite{H})}
 Let $q=p^k$, $d=\gcd(q+1,3)$. The following is the list of maximal subgroups of $\PSU(3,q)$ up to conjugacy:
\begin{itemize}
\item[i)] the stabilizer of a $\F_{q^2}$-rational point of $\cH_q$, of order $q^3(q^2-1)/d$;
\item[ii)] the stabilizer of a $\F_{q^2}$-rational point off $\cH_q$ and its polar line (which is a $(q+1)$-secant to $\cH_q$), of order $q(q-1)(q+1)^2/d$;
\item[iii)] the stabilizer of the self-polar triangle, or order $6(q+1)^2/d$;
\item[iv)] the normalizer of a cyclic Singer group stabilizing a triangle in $\PG(2,q^6)\setminus\PG(2,q^2)$, of order $3(q^2-q+1)/d$;

{\rm for $p>2$:}
\item[v)] $\PGL(2,q)$ preserving a conic;
\item[vi)] $\PSU(3,p^m)$ with $m\mid k$ and $k/m$ odd;
\item[vii)] subgroups containing $\PSU(3,2^m)$ as a normal subgroup of index $3$, when $m\mid k$, $k/m$ is odd, and $3$ divides both $k/m$ and $q+1$;
\item[viii)] the Hessian groups of order $216$ when $9\mid(q+1)$, and of order $72$ and $36$ when $3\mid(q+1)$;
\item[ix)] $\PSL(2,7)$ when $p=7$ or $-7$ is not a square in $\mathbb{F}_q$;
\item[x)] the alternating group $\mathbf{A}_6$ when either $p=3$ and $k$ is even, or $5$ is a square in $\mathbb{F}_q$ but $\mathbb{F}_q$ contains no cube root of unity;
\item[xi)] the symmetric group $\mathbf{S}_6$ when $p=5$ and $k$ is odd;
\item[xii)] the alternating group $\mathbf{A}_7$ when $p=5$ and $k$ is odd;

{\rm for $p=2$:}
\item[xiii)] $\PSU(3,2^m)$ with $m\mid k$ and $k/m$ an odd prime;
\item[xiv)] subgroups containing $\PSU(3,2^m)$ as a normal subgroup of index $3$, when $k=3m$ with $m$ odd;
\item[xv)] a group of order $36$ when $k=1$.
\end{itemize}
\end{theorem}

The previous theorem will be used for a case-analysis of the possible unitary groups $G$ such that the quotient curve $\cH/G$ realizes the Galois covering.

While dealing with case \textit{ii)}, we will invoke a result by Dickson \cite{D} which classifies all subgroups of the projective special linear group $\PSL(2,q)$ acting on $\PG(1,q)$. We remark that $\PSL(2,q)$ has index $\gcd(q-1,2)$ in the group $\PGL(2,q)$ of all projectivities of $\PG(1,q)$. From Dickson's result the classification of subgroups of $\PGL(2,q)$ is easily obtained.

\begin{theorem}{\rm (\cite[Chapt. XII, Par. 260]{D}; see also \cite[Thm. A.8]{HKT})}\label{Di}
Let $q=p^k$, $d=\gcd(q-1,2)$. The following is the complete list of subgroups of $\PGL(2,q)$ up to conjugacy:
\begin{itemize}
\item[i)] the cyclic group of order $h$ with $h\mid(q\pm1)$;
\item[ii)] the elementary abelian $p$-group of order $p^f$ with $f\leq k$;
\item[iii)] the dihedral group of order $2h$ with $h\mid(q\pm1)$;
\item[iv)] the alternating group $\mathbf{A}_4$ for $p>2$, or $p=2$ and $k$ even;
\item[v)] the symmetric group $\mathbf{S}_4$ for $16\mid(q^2-1)$;
\item[vi)] the alternating group $\mathbf{A}_5$ for $p=5$ or $5\mid(q^2-1)$;
\item[vii)] the semidirect product of an elementary abelian $p$-group of order $p^f$ by a cyclic group of order $h$, with $f\leq k$ and $h\mid(q-1)$;
\item[viii)] $\PSL(2,p^f)$ for $f\mid k$;
\item[ix)] $\PGL(2,p^f)$ for $f\mid k$.
\end{itemize}
\end{theorem}
%
%In order to prove Theorem \ref{GGS} we need the following results on the curves $\cC_q$ and $\cX_q$.
%
%\begin{theorem}{\rm (\cite{GOS},\cite{GMP})}\label{GGSpuntofisso}
%For $n\geq5$, the group $\aut(\cC_q)$ has a unique fixed point $P_\infty$ on $\cC_q$, and $P_\infty$ is $\fqs$-rational.
%\end{theorem}
%
%\begin{corollary}\label{cor1}
%Let $G\leq\aut(\cH_q)$. If there exists $\bar G\leq\aut(\cH_q)$ such that $G$ is a proper normal subgroup of $\bar G$ and $\bar G$ acts semiregularly on $\cH_q$, then $\bar G/G\leq\aut(\cH_q/G)$ acts semiregularly on $\cH_q/G$, hence $\cC_q\not\cong\cH_q/G$.
%\end{corollary}
%
%\begin{theorem}{\rm (\cite[Thm. 12.11]{HKT})} The group $\aut(\cX_q)$ has a unique fixed point $Q_\infty$ on $\cX_q$, and $Q_\infty$ is $\fqs$-rational.
%\end{theorem}
%
%\begin{corollary}
%Let $G\leq\aut(\cH_{q^3})$. If there exists $\bar G\leq\aut(\cH_{q^3})$ such that $G$ is a proper normal subgroup of $\bar G$ and $\bar G$ acts semiregularly on $\cH_{q^3}$, then $\bar G/G\leq\aut(\cH_{q^3}/G)$ acts semiregularly on $\cH_{q^3}/G$, hence $\cX_q\not\cong\cH_{q^3}/G$.
%\end{corollary}

\section{$\cC_{2^n}$ is not Galois-covered by $\cH_{2^n}$, for any $n\geq5$}

The aim of this section is to prove Theorem \ref{result1}. Throughout the section, let $n\geq5$ be odd and $q=2^n$. 
We rely on the following result by Duursma and Mak.

\begin{lemma}\label{gradofisso}{\rm (\cite[Thm. 1.2]{DM})}
Let $n\geq5$ odd. If $\cC_{2^n}\cong\cH_{2^n}/G$ for some $G\leq\aut(\cH_{2^n})$, then $G$ has order $(2^n+1)/3$ and acts semiregularly on $\cH_{2^n}$.
\end{lemma}

By Lemma \ref{gradofisso} only subgroups $G$ of $\aut(\cH_{q})$ of order $(q+1)/3$ acting semiregularly on $\cH_{q}$ need to be considered. We will also use the fact that the whole automorphism group of $\aut(\cC_{2^n})$ fixes a point.

\begin{theorem}{\rm (\cite[Thm. 3.10]{GOS},\cite[Prop. 2.10]{GMP})}\label{GGSpuntofisso}
For $n\geq5$, the group $\aut(\cC_{2^n})$ has a unique fixed point $P_\infty$ on $\cC_q$, and $P_\infty$ is $\fqs$-rational.
\end{theorem}

\begin{corollary}\label{cor1}
Let $G\leq\aut(\cH_q)$. If there exists $\bar G\leq\aut(\cH_q)$ such that $G$ is a proper normal subgroup of $\bar G$ and $\bar G$ acts semiregularly on $\cH_q$, then $\bar G/G\leq\aut(\cH_q/G)$ acts semiregularly on $\cH_q/G$, hence $\cC_{2^n}\not\cong\cH_q/G$.
\end{corollary}

The following well-known result about finite groups  will be used (see for example \cite[Ex. 16 Page 232]{Mac}).

\begin{lemma}\label{indice}
Let $H$ be a finite group and $K$ a subgroup of $H$ such that the index $[H:K]$ is the smallest prime number dividing the order of $H$. Then $K$ is normal in $H$.
\end{lemma}

\begin{proposition}\label{punto-retta}
Suppose $G\leq\PSU(3,q)$ and a maximal subgroup of $\PSU(3,q)$ containing $G$ satisfies case $\mathit{ii)}$ in Theorem \ref{MH}. Then $\cC_{2^n}\not\cong\cH_q/G$.
\end{proposition}

\begin{proof}
Let $\ell$ be the $(q+1)$-secant to $\cH_q$ stabilized by $G$; we show that $G$ is isomorphic to a cyclic subgroup of $\PSL(2,q^2)$.

$\PGU(3,q)$ is transitive on the points of $\PG(2,q^2)\setminus\cH_q$ (see for example \cite{HP}), hence also on the $(q+1)$-secant lines; therefore we can assume that $\ell$ is the line at infinity $T=0$. The action on $\ell$ of an element $g\in G$ is given by $(X,Y,0)\mapsto A_g\cdot(X,Y,0)$, where the matrix $A_g=(a_{ij})_{i=1,2,3}^{j=1,2,3}$ satisfies $a_{31}=a_{32}=0$, and we can assume $a_{33}=1$. By direct computation, it is easy to check that the application
$$ \varphi:G\rightarrow\PGL(2,q^2), \qquad \varphi(g): \begin{pmatrix} X \\ Y \end{pmatrix} \mapsto \begin{pmatrix} a_{11} & a_{12} \\ a_{21} & a_{22} \end{pmatrix} \cdot \begin{pmatrix} X \\ Y \end{pmatrix}, $$
is a well-defined group homomorphism. Moreover, $\varphi$ is injective, since no non-trivial element of $G$ can fix the points of $\cH_q\cap\ell$, by the semiregularity of $G$. Hence $G$ is isomorphic to a subgroup of $\PGL(2,q^2)\cong\PSL(2,q^2)$.
Since $|G|$ is odd, then Theorem \ref{Di} implies that $G$ is cyclic.

Let $g\in G$ be an element of prime order $d>3$; such a $d$ exists, since it is easy to check that $2^n+1$ is a power of $3$ only when $n=1$ or $n=3$. If we denote by $d^h$ the highest power of $d$ dividing $(q+1)/3$, then $d^{2h}$ is the highest power of $d$ dividing $$|\PGU(3,q)|=q^3(q^3+1)(q^2-1)=q^3(q+1)^2(q-1)(q^2-q+1).$$
Let $\cH_q$ have equation $X^{q+1}+Y^{q+1}+T^{q+1}=0$, then $$ D=\left\{ (X,Y,T)\mapsto(\lambda X,\mu Y,T) \mid \lambda^{d^h}=\mu^{d^h}=1 \right\} $$
is a Sylow $d$-subgroup of $\PGU(3,q)$, and by Sylow's theorems we can assume up to conjugation that $g\in D$, so the fixed points of the subgroup $\langle g\rangle$ generated by $g$ are the fundamental points $P_i$, $i=1,2,3$.
Since $G$ is abelian, then $\la g\ra$ is normal in $G$, hence $G$ acts on the fixed points $T=\left\{P_1,P_2,P_3\right\}$ of $\langle g\rangle$. In fact, for all $k\in G$ and $\bar g\in\la g\ra$,
$$k(P_i)=k(\bar g(P_i))=\tilde g(k(P_i))$$
for some $\tilde g\in\la g\ra$, that is, $k(P_i)$ is fixed by $\tilde g$, hence $k(P_i)$ is a fundamental point $P_j$.

As $|G|$ is odd, we have by the orbit stabilizer theorem that the orbits of any $k\in G$ on $T$ have length $1$ or $3$. If $k$ has a single orbit on $T$, then the matrix representing $k$ is
$$ k=\begin{pmatrix} 0 & 0 & \lambda \\ \mu & 0 & 0 \\ 0 & \rho & 0  \end{pmatrix} \quad{\rm or}\quad k=\begin{pmatrix} 0 & \lambda & 0 \\ 0 & 0 & \mu \\ \rho & 0 & 0  \end{pmatrix},\quad{\rm in\; both\; cases}\quad k^3=\begin{pmatrix} \lambda\mu\rho & 0 & 0 \\ 0 & \lambda\mu\rho & 0 \\ 0 & 0 & \lambda\mu\rho \end{pmatrix},  $$
that is $k^3=1$, hence $G$ cannot be generated by $k$. Therefore a generator $\alpha$ of $G$ has the form
$$ \alpha:(X,Y,T)\mapsto(\theta X,\eta Y,T), $$
with $\theta^\frac{q+1}{3}=\eta^\frac{q+1}{3}=1$. If $\theta$ had order $m<(q+1)/3$, then $\alpha^m$ would fix the points of $\cH_q\cap(Y=0)$, against the semiregularity of $G$. Then $\theta$ is a primitive $(q+1)/3$-th root of unity, and the same holds for $\eta$, so that
\begin{equation}\label{alpha}
\alpha=\alpha_\theta:(X,Y,T)\mapsto(\theta X, \theta^i Y,T),
\end{equation}
where $\theta$ is a primitive $(q+1)/3$-th root of unity, and $i$ is co-prime with $(q+1)/3$.

Let $\zeta\in\fqs$ satisfy $\zeta^3=\theta$, and let $\bar G$ be the group generated by $\alpha_\zeta$, as defined in \eqref{alpha}. Any element of $\bar G$ fixes only the fundamental points, hence $\bar G$ is semiregular on $\cH_q$; moreover, $G$ is normal in $\bar G$ of index $3$. Then Corollary \ref{cor1} yields the thesis.
\qed
\end{proof}

\begin{proposition}\label{triangolo}
Suppose $G\leq\PSU(3,q)$ and a maximal subgroup of $\PSU(3,q)$ containing $G$ satisfies case $\mathit{iii)}$ in Theorem \ref{MH}. Then $\cC_{2^n}\not\cong\cH_q/G$.
\end{proposition}

\begin{proof}
Up to conjugation, the self-polar triangle stabilized by $G$ is the fundamental triangle $T=\left\{P_1,P_2,P_3\right\}$. Let $N$ be the subgroup of $G$ stabilizing $T$ pointwise. Then $N$ is normal in $G$, since $g^{-1}ng(P_i)=g^{-1}n(g(P_i))=g^{-1}(g(P_i))=P_i$, where $n\in N$, $g\in G$. The group $G/N$ acts faithfully on $T$, hence either $G=N$ or $[G:N]=3$.

If $G=N$, then $G$ fixes one fundamental point $P_i$, which is off $\cH_q$, and the polar line of $P_i$ passing through the other fundamental points; therefore Proposition \ref{punto-retta} yields the thesis.

Now suppose $[G:N]=3$. As in the proof of Proposition \ref{punto-retta}, $N$ is isomorphic to a subgroup of $\PSL(2,q^2)$; since $|N|$ is odd, we have by Theorem \ref{Di} that $N$ is cyclic, say $N=\langle\alpha_\xi\rangle$, where $\xi$ is a primitive $(q+1)/9$-th root of unity and $\alpha_\xi$ is defined in \eqref{alpha}. Let $h\in G\setminus N$. By arguing as for $k$ in the proof of Proposition \ref{punto-retta}, we have that $h$ has order $3$. Moreover, $G$ is the semidirect product $N\rtimes\langle h\rangle$, because $N\triangleleft G$, $N\cap\langle h\rangle=\emptyset$, and the orders of the subgroups imply $G=\langle h\rangle\cdot N$.

Let $\bar N$ be the cyclic group $\langle \alpha_\theta\rangle$, with $\theta\in\fqs$ such that $\theta^3=\xi$%; note that $N=\bar N^3=\left\{x^3\mid x\in\bar N\right\}$.
, and let $\bar G$ be the group generated by $\bar N$ and $h$. $\bar G$ is the semidirect product $\bar N\rtimes\langle h\rangle$; in fact, $\bar N$ is normal in $\bar G$ by Lemma \ref{indice}, $\bar N\cap\langle h\rangle=\emptyset$, and the orders of the subgroups imply $\bar G=\bar N\cdot \langle h\rangle$. We have that $G$ is normal in $\bar G$, again by Lemma \ref{indice}.

%Let $\bar G$ be the group generated by $\bar N$ and $h$; in order to prove that $\bar G$ is the semidirect product $\bar N\rtimes\langle h\rangle$ it is enough to show that $\bar N$ is normal in $\bar G$, since $\bar N\cap\langle h\rangle=\emptyset$ and $\bar G=\langle h\rangle\cdot\bar N$ by the orders of the subgroups. Any element of $\bar G\setminus \bar N$ has the form $h\bar n $ or $h^2\bar n $, with $\bar n\in\bar N$. Consider $h\bar n $ (the same argument holds for $h^2\bar n $), we have $$ \left[(h\bar n)\alpha_\theta(h\bar n)^{-1}\right]^3 = \left[h(\bar n\alpha_\theta\bar n^{-1})h^{-1}\right]^3 = \left[h\alpha_\theta h^{-1}\right]^3 = h\alpha_\theta^3 h^{-1} = h\alpha_\xi h^{-1} \in N $$ since $N$ is normal in $G$; then $(h\bar n)\alpha_\theta(h\bar n)^{-1}\in\bar N$, and the same argument holds for $h^2\bar n$. Then $\bar N$ is normal in $\bar G$, and $\bar G=\bar N\rtimes\langle h\rangle$.
We want to count in two ways the size of the set
$$ I=\left\{ (\bar g,P)\mid \bar g\in\bar G\setminus\left\{id\right\},\;P\in\cH_q\,,\;\bar g(P)=P \right\} $$

The diagonal group $\bar N$ is semiregular on $\cH_q$, like $\bar G\cap\PSU(3,q)=G$. Then we consider only elements of the form $\bar n h$ or $\bar n h^2$, with $\bar n\in\bar N\setminus N$.% Recall that $h$ acts as a $3$-cycle on the fundamental points, hence we can assume that it has the matricial form
We have
$$ \bar n= \begin{pmatrix} \rho & 0 & 0 \\ 0 & \rho^i & 0 \\ 0 & 0 & 1 \end{pmatrix},\qquad  h=\begin{pmatrix} 0 & \lambda & 0 \\ 0 & 0 & \mu \\ 1 & 0 & 0 \end{pmatrix} $$
where $\lambda^{q+1}=\mu^{q+1}=1$, $\gcd(i,(q+1)/3)=1$, $\rho=\theta^j$ with $0\leq j<(q+1)/3$ (the argument is analogous in case $h$ acts as the other possible $3$-cycle on the fundamental points). Hence $\bar n h$ is
$$ \bar n h = \begin{pmatrix} \rho & 0 & 0 \\ 0 & \rho^i & 0 \\ 0 & 0 & 1 \end{pmatrix} \cdot \begin{pmatrix} 0 & \lambda & 0 \\ 0 & 0 & \mu \\ 1 & 0 & 0 \end{pmatrix} = \begin{pmatrix} 0 & A & 0 \\ 0 & 0 & B \\ 1 & 0 & 0 \end{pmatrix}, $$
where $A^{q+1}=B^{q+1}=1$, and $\det(\bar nh)=AB$ is not a cube in $\fqs$, since $\bar nh\notin\PSU(3,q)$.The eigenvalues of $\bar n h$ are the zeros of $X^3-AB\in\fqs[X]$. Since $\fqs$ has characteristic $2$, we get $3$ distinct eigenvalues in a cubic extension of $\fqs$, namely $z$, $zx$, and $z(x+1)$, where $x^2+x+1=0$ and $z^3=AB$. Then $\bar n h$ has exactly $3$ fixed points, given by $3$ independent eigenvectors:
$$ Q_1=\left(z,\frac{z^2}{A},1\right),\quad Q_2=\left(zx,\frac{z^2x^2}{A},1\right),\quad Q_3=\left(z(x+1),\frac{z^2(x+1)^2}{A},1\right). $$
$Q_1$ is a point of $\cH_q$. In fact, since $\cH_q$ has equation $X^{q+1}+Y^{q+1}+T^{q+1}=0$, then
$$ z^{q+1}+\left(\frac{z^2}{A}\right)^{q+1}+1 = z^{q+1}+z^{2(q+1)}+1 = \frac{(z^{q+1})^3-1}{z^{q+1}-1} = \frac{A^{q+1}-1}{z^{q+1}-1}=0 $$
as $z\notin\fqs$ implies $z^{q+1}\neq 1$. Similarly we get $Q_2\in\cH_q$, $Q_3\in\cH_q$.

Then each element $\bar n h$ or $\bar n h^2$ with $\bar n\in\bar N\setminus N$ has exactly $3$ fixed points on $\cH_q$, and
\begin{equation}\label{I}
|I|=2\cdot\left(|\bar{N}|-|N|\right)\cdot3=2\cdot\left(\frac{q+1}{3}-\frac{q+1}{9}\right)\cdot3=4\cdot\frac{q+1}{3}.
\end{equation}
The orbit $\cO$ under $\bar G$ of a point $P\in\cH_q$ contains the orbit of $P$ under $G$, hence $|\cO|\geq(q+1)/3$; by the orbit stabilizer theorem, the stabilizer $\cS$ of $P$ under $\bar G$ has size $|\cS|\leq3$, in particular $|\cS|\in\left\{1,3\right\}$ since $|\bar G|$ is odd. Then $|I|=2m$, where $m$ is the number of points of $\cH_q$ which are fixed by some non-trivial element of $\bar G$. By \eqref{I}, we get
$$ m=2\cdot\frac{q+1}{3},$$
that is, these $m$ points form $2$ distinct orbits under the action of $G$. Then the quotient group $\bar G/G$ has $2$ fixed points on $\cH_q/G$ and any other orbit of $\bar G/G$ is long, with length $3$.

By \ref{GGSpuntofisso}, one of the fixed points of $\bar G/G$ is $\fqs$-rational, and the other one may or may not be $\fqs$-rational. Then the number of $\fqs$-rational points of $\cH_q/G$ is congruent to $1$ or $2$ mod $3$.

On the other side, the $\fqs$-maximal curve $\cC_{2^n}$ has genus $g=(3q-4)/2$ and number of $\fqs$-rational points equal to
$$ |\cC_{2^n}(\fqs)| = q^2+1+2qg = q^2+1+2q\cdot(3q-4)/2 = 4q^2-4q+1 ,$$
see \cite[Prop. 2.2]{GGS}; then $|\cC_{2^n}(\fqs)|\equiv0\,(\mod\,3)$, as $q\equiv2\,(\mod\,3)$. Therefore, $\cH_q/G\not\cong\cC_{2^n}$.
\qed
\end{proof}

\begin{proposition}\label{puntorettafuori}
Suppose $G\not\subseteq\PSU(3,q)$ and a maximal subgroup of $\PSU(3,q)$ containing $G\cap\PSU(3,q)$ satisfies case $\mathit{ii)}$ in Theorem \ref{MH}. Then $\cC_{2^n}\not\cong\cH_q/G$.
\end{proposition}

\begin{proof}
Let $G'=G\cap\PSU(3,q)$. Since $\PSU(3,q)$ has prime index $3$ in $\PGU(3,q)$, then $\PGU(3,q)=G\cdot\PSU(3,q)$, hence $[G:G']=3$, and $G'$ is normal in $G$ by Lemma \ref{indice}.

Arguing as in the proof of Proposition \ref{punto-retta}, $G'=\la\alpha_\xi\ra$ is cyclic, where $\xi$ is a primitive $(q+1)/9$-th root of unity, $\alpha_\xi$ is defined in \eqref{alpha} and fixes the fundamental points, and $G$ stabilizes the fundamental triangle $T$.

Suppose there exists $h\in G\setminus G'$ of order $3$. Arguing as for $k$ in Proposition \ref{triangolo}, $G=G'\rtimes\la h\ra$. Let $\theta\in\fqs$ with $\theta^3=\xi$, we define $\bar{G'}$ as the cyclic group generated by $\alpha_\theta$ (given in \eqref{alpha}), and $\bar G$ as the group generated by $\bar{G'}$ and $h$.
Again, it is easily seen that $\bar G=\bar{G'}\rtimes\la h\ra$; moreover, $G'$ is normal in $\bar{G'}$ and $G$ is normal in $\bar G$ with indeces $[\bar G:G]=[\bar{G'}:G']=3$. We can repeat the same argument as in the proof of Proposition \ref{triangolo} after replacing $N$ with $G'$ and $\bar N$ with $\bar{G'}$; in this way we obtain that $|\cH_q/G(\fqs)|\equiv1,2\,(\mod\,3)$, while $|\cC_{2^n}|\equiv0\,(\mod\,3)$. This yields the thesis.

Now suppose there is no $h\in G\setminus G'$ of order $3$. This fact implies that $G$ is made of diagonal matrices, since $G$ acts on $T$. Then, by Theorem \ref{Di}, $G$ is cyclic and $G=\la\alpha_\theta\ra$, where the notations are the same as above. We define the diagonal group $\bar G=\la\alpha_\zeta\ra$, with $\zeta^3=\theta$. $G$ is normal in $\bar G$ of index $3$ and$\bar G$ is semiregular on $\cH_q$, hence Corollary \ref{cor1} yields the thesis.
\qed
\end{proof}

\begin{proposition}\label{triangolofuori}
Suppose $G\not\subseteq\PSU(3,q)$ and a maximal subgroup of $\PSU(3,q)$ containing $G\cap\PSU(3,q)$ satisfies case $\mathit{iii)}$ in Theorem \ref{MH}. Then $\cC_{2^n}\not\cong\cH_q/G$.
\end{proposition}

\begin{proof}
As above, $G'=G\cap\PSU(3,q)$ is normal in $G$ of index $3$. By applying to $G'$ the argument of Proposition \ref{triangolo}, we get that either $G'$ is cyclic and $G'=\la\alpha_\xi\ra$, or $G'=\la\alpha_\eta\ra\rtimes\la h\ra$, where $\eta$ is a primitive $(q+1)/27$-th root of unity and $h$ is an element of order $3$ acting as a $3$-cycle on the fundamental triangle $T$.

Consider the case $G'=\la\alpha_\xi\ra$. Since $G'$ is normal in $G$, then $G$ acts on $T$. If $G$ fixed $T$ pointwise, then $G$ would be made of diagonal matrices whose non-zero coefficients are cubes in $\fqs$ being $(q+1)/3$-th roots of unity, hence $G\leq\PSU(3,q)$, against the hypothesis. Then, arguing as above, $G=G'\rtimes\la h\ra$, where $h\in G\setminus G'$ has order $3$. Let $\theta\in\fqs$ with $\theta^3=\xi$, and define $\bar G=\la\alpha_\theta\ra\rtimes\la h\ra$.
Arguing as in Proposition \ref{triangolo}, we obtain that $|\cH_q/G(\fqs)|\equiv1,2\,(\mod\,3)$, while $|\cC_{2^n}|\equiv0\,(\mod\,3)$. This yields the thesis.

Now consider the case $G'=\la\alpha_\eta\ra\rtimes\la h\ra$. $\la\alpha_\eta\ra$ is the only subgroup of $G'$ of order $(q+1)/27$, then $\la\alpha_\eta\ra$ is a characteristic subgroup of $G'$; also, $G'$ is normal in $G$. Therefore $\la\alpha_\eta\ra$ is normal in $G$, hence $G$ acts on the fundamental points.
Let $G''$ be the subgroup of $G$ fixing $T$ pointwise; $G''$ is normal in $G$ of index $3$, and $G=G''\rtimes\la h\ra$. Being made of diagonal matrices, $G''$ is abelian, with a subgroup $G'$ of index $3$. By the primary decomposition of abelian groups, we have either $G''=\la\alpha_\xi\ra$ with $\xi^3=\eta$, or $G''=\la\alpha_\eta\ra\times\la h'\ra$, with $h'\in G''\setminus\la\alpha_\eta\ra$ a diagonal matrix of order $3$.
In the latter case, by ${h'}^3=id$ we get that $\det(h')^3=1$ and then $\det(h')$ is a cube in $\fqs$, hence $h'\in G\cap\PSU(3,q)=G'$; therefore $G'=G''$, against the fact that $h$ is a $3$-cycle on $T$.

Then $G''=\la\alpha_\xi \ra$, and $G=\la\alpha_\xi \ra\rtimes\la h\ra$. Let $\bar G=\la\alpha_\theta\ra\rtimes\la h\ra$, where $\theta\in\fqs$ satisfies $\theta^3=\xi$. We can repeat the same argument as in the proof of Proposition \ref{triangolo} after replacing $N$ with $\la\alpha_\xi \ra$ and $\bar N$ with $\la\alpha_\theta\ra$; in this way we obtain that $|\cH_q/G(\fqs)|\equiv1,2\,(\mod\,3)$, while $|\cC_{2^n}|\equiv0\,(\mod\,3)$.
\qed
\end{proof}

\begin{lemma}\label{lemmino}
Suppose $G\leq\PSU(3,q)$ and any maximal subgroup $M$ of $\PSU(3,q)$ containing $G$ does not satisfy case $\mathit{ii)}$ nor case $\mathit{iii)}$ in Theorem \ref{Di}. Then $M$ satisfies only case $\mathit{xiv)}$, i.e. $G\not\subseteq\PSU(3,2^m)$ and $M$ contains $\PSU(3,2^m)$ as a normal subgroup of index $3$, where $n=3m$.
\end{lemma}

\begin{proof}
We can exclude cases $ii)$ and $iii)$ by hypothesis, case $i)$ by the semiregularity of $G$, and cases $iv)$ and $xv)$ because $|G|$ is not a divisor of their orders. Then the thesis follows if we exclude case $xiii)$. To this end, we apply again Theorem \ref{MH} to $\PSU(3,2^m)$, where $n=p'm$ with $p'\geq3$ an odd prime. Note that since $n\geq5$ is odd, then either $p'\geq5$, or $p'=3$ and $m\geq5$.

Case $i)$. $G$ fixes a point $P\in\cH_{2^m}$. Then $P\notin\cH_q$ by the semiregularity of $G$ on $\cH_q$, hence $G$ satisfies case $ii)$ in the list of maximal subgroups of $\PSU(3,q)$, contradicting the hypothesis.

Case $ii)$. By Lagrange's theorem, the order $(2^{p'm}+1)/3$ of $G$ divides $2^m(2^m-1)(2^m+1)^2/3$, hence $\sum_{i=0}^{p'-1}(-1)^i 2^{im}$ divides $(2^{2m}-1)$, which is impossible for any odd $p'\geq3$.

Case $iii)$. Now $(2^{p'm}+1)/3$ divides $2(2^m+1)^2$, hence $\sum_{i=0}^{p'-1}(-1)^i 2^{im}$ divides $3(2^m+1)$, which is impossible since $\sum_{i=0}^{p'-1}(-1)^i 2^{im}>3(2^m+1)$.

Case $iv)$. Now $(2^{p'm}+1)/3$ divides $(2^{2m}-2^m+1)$, which is impossible since $(2^{p'm}+1)/3>2^{2m}-2^m+1$ for any $p'\geq3$, $m\geq3$.

Case $xiii)$. $G$ is contained in $\PSU(3,2^r)$ with $m/r=p''\geq3$ an odd prime, and $n/r=p'p''\geq9$. This is impossible since $|G|$ is greater than the order of any maximal subgroup of $\PSU(3,2^r)$.

Case $xiv)$. $G$ is contained in a group $K$ containing $\PSU(3,2^r)$ as a normal subgroup of index $3$, where $r=m/3$. If $H\neq\PSU(3,2^r)$ is a maximal subgroup of $K$, we have $H\cdot\PSU(3,2^r)=\PGU(3,2^r)$, hence $[H:H\cap\PSU(3,2^r)]=[\PGU(3,2^r):\PSU(3,2^r)]=3$. Therefore, $|H|/3$ divides the order of a maximal subgroup of $\PSU(3,2^r)$. Then we get a contradiction, since $|G|$ does not divide three times the order of any maximal subgroup of $\PSU(3,2^r)$.

Case $xv)$. $|G|$ divides $36$, and $m=1$, which implies $p'\geq5$. For $p'=5$, we have $|G|=11$ which does not divide $36$; for $p'>5$, we have that $|G|$ is greater than $36$.
\qed
\end{proof}

\begin{proposition}\label{quattordici}
Suppose $G\leq\PSU(3,q)$ and a maximal subgroup $M$ of $\PSU(3,q)$ containing $G$ satisfies only case $xiv)$ in Theorem \ref{Di}. Then $\cC_{2^n}\not\cong\cH_q/G$.
\end{proposition}

\begin{proof}
$M$ contains $\PSU(3,2^m)$ as a normal subgroup of order $3$, where $m=n/3$. Arguing as in case $xiv)$ of Lemma \ref{lemmino}, $|G|$ divides three times the order of a maximal subgroup of $\PSU(3,2^m)$. Then we multiply by $3$ the orders of maximal subgroups of $\PSU(3,q)$ as listed in Theorem \ref{Di}.

Case $i)$. The order $(2^{3m}+1)/3$ of $G$ divides $2^{3m}(2^{2m}-1)$, hence $(2^{2m}-2^m+1)$ divides $3(2^m-1)$, which is impossible since $m\geq3$.

Case $ii)$. $(2^{3m}+1)/3$ divides $2^m(2^m+1)^2(2^m-1)$, which is impossible as above.

Case $iii)$. $(2^{3m}+1)/3$ divides $6(2^m+1)^2$, hence $(2^{2m}-2^m+1)$ divides $9(2^m+1)$, which is impossible for any $m\geq3$.

Case $iv)$. $(2^{3m}+1)/3$ divides $3(2^{2m}-2^m+1)$, hence $(2^m+1)\mid9$, which implies $m=3$.

Cases $xiii)$ and $xiv)$. $(2^{3m}+1)/3$ divides either $3\cdot|PSU(3,2^r)|$ or $3\cdot|PGU(3,2^r)|$, where $m/r=p''\geq3$ is an odd prime. As in the proof of Lemma \ref{lemmino}, this is impossible since $|G|$ exceeds three times the order of any subgroup of $\PGU(3,2^r)$.

Case $xv)$. $(2^{3m}+1)/3$ divides $36$, which is impossible for any $m\geq3$.

Therefore the only possibility is given in case $iv)$ for $m=3$. Then $G$ has order $171$, $G''=G\cap\PSU(3,2^m)$ has order $|G|/3=57$, and $G''$ is contained in the normalizer $N$ of a cyclic Singer group $S$, of order $|N|=(2^{2m}-2^m+1)=57$, hence $G''=N$. $G''$ acts on the three non-collinear points $Q_1,Q_2,Q_3$ fixed by $S$, whose coordinates are in a cubic extension of $\mathbb{F}_{2^{2m}}$, hence in $\mathbb{F}_{2^{2n}}$. By the semiregularity of $G$, we have $Q_i\notin\cH_q$; then $\left\{Q_1,Q_2,Q_3\right\}$ is a self-polar triangle, and we get the thesis as in the proof of Proposition \ref{triangolofuori}, after replacing $q$ with $2^m$ and $G'$ with $G''$.
\qed
\end{proof}

\begin{theorem}
$\cC_{2^n}$ is not a Galois subcover of the Hermitian curve $\cH_q$.
\end{theorem}

\begin{proof}
Suppose $\cC_{2^n}\cong\cH_q/G$. By Propositions \ref{punto-retta}, \ref{triangolo}, \ref{quattordici} and Lemma \ref{lemmino}, we have that $G\not\subseteq\PSU(3,q)$; then $G'=G\cap\PSU(3,q)$ has index $3$ in $G$.
After replacing $G$ with $G'$, we can repeat the proofs of Propositions \ref{puntorettafuori} and \ref{triangolofuori}, the proof of Lemma \ref{lemmino}, and the first part of the proof of Proposition \ref{quattordici}. In this way, the only possibility we have is that $n=9$ and a maximal subgroup $M$ of $\PSU(3,2^9)$ containing $G'$ contains $\PSU(3,2^3)$ as a normal subgroup of index $3$; moreover, $G''=G'\cap\PSU(3,2^3)$ is contained in the normalizer $N'$ of a cyclic Singer group, of order $|N'|=57$.

If $G'\leq\PSU(3,2^3)$, then we repeat the argument of the proof of Proposition \ref{quattordici}, after $G$ with $G'$. In this way we get a contradiction.

If $G'\not\subseteq\PSU(3,2^3)$, then $G''=G'\cap\PSU(3,2^3)$ has order $|G'|/3=19$. Since $|G'|=57$, then by the third Sylow theorem $G'$ is the only Sylow $19$-subgroup of $G$, hence $G''$ is the cyclic Singer group normalized by $G'=N'$. Therefore $G''$ fixes a triangle with coordinates in the cubic extension $\mathbb{F}_{2^{18}}$ of $\mathbb{F}_{2^{6}}$, which is the fundamental triangle $T$ up to conjugation in $\PGU(3,2^9)$. Hence $G'$ acts on $T$, and Proposition \ref{triangolofuori} yields the thesis.
\end{proof}

\section{$\cX_{q}$ is not Galois-covered by $\cH_{q^3}$, for any $q>3$}

The aim of this section is to prove Theorem \ref{result2}.
Throughout the section, let $q>3$ be a power of a prime $p$.
%We will use the following result.
%
%\begin{theorem}{\rm (\cite[Thm. 12.11]{HKT})} The group $\aut(\cX_q)$ has a unique fixed point $Q_\infty$ on $\cX_q$, and $Q_\infty$ is $\fqs$-rational.
%\end{theorem}
%
%\begin{corollary}
%Let $G\leq\aut(\cH_{q^3})$. If there exists $\bar G\leq\aut(\cH_{q^3})$ such that $G$ is a proper normal subgroup of $\bar G$ and $\bar G$ acts semiregularly on $\cH_{q^3}$, then $\bar G/G\leq\aut(\cH_{q^3}/G)$ acts semiregularly on $\cH_{q^3}/G$, hence $\cX_q\not\cong\cH_{q^3}/G$.
%\end{corollary}

%In the following proposition we bound the degree of a possible Galois-covering $\cH_{q^3}\rightarrow\cX_q$.
By direct application of a result by Duursma and Mak, we have the following bound.

\begin{proposition}{\rm (\cite[Thm. 1.3]{DM})}\label{possibilivalori}
If there exists a Galois-covering $\cH_{q^3}\rightarrow\cX_q$, of degree $d$, then $$q^2+q\leq d\leq q^2+q+2.$$
\end{proposition}

%\begin{proof}
%Since $\cX_q$ has genus $g(\cX_q)=(q-1)(q^3-q)/2$, then $2g(\cX_q)-2=A(q^3+1)-B$, with $A=q-1$ and $B=q^2+1$. Then we can apply the bound in \cite[Proposition 5.1]{DM}, so that
%$$ d\geq\frac{(q+1)(q^3+1)}{q^2+1} = q^2+q-1+\frac{2}{q^2+1}, $$
%hence $d\geq q^2+q$. By the Hurwitz genus formula (see \cite[Theorem 3.4.13]{Sti}), we have
%$$ d\leq\frac{2g(\cH_{q^3})-2}{2g(cX_q)-2} = \frac{q^6-q^3-2}{q^4-q^3-q^2+q-2} = q^2+q+2+\frac{q^3+3q^2+2}{q^4-q^3-q^2+q-2}, $$
%hence $d\leq q^2+q+2$.
%\qed
%\end{proof}

Therefore we have to exclude three possible values of $d$.

\begin{proposition}\label{primovalore}
There is no Galois-covering $\cH_{q^3}\rightarrow\cX_q$ of degree $q^2+q+2$.
\end{proposition}

\begin{proof}
If such a Galois-covering exists, then $q^2+q+2$ divides the order $q^9(q^9+1)(q^6-1)$ of $\PGU(3,q^3)$, hence $q^2+q+2$ divides the remainder of the polynomial division, which is equal to $2128q-1568$. Then the possible values for $q$ are $1$, $2$, $3$, or $10$, but none of these is acceptable.
\qed\end{proof}
\vspace*{.2cm}
Now we consider the possible value $d=q^2+q+1$.

\begin{lemma}\label{lem}
Let $G\leq\PGU(3,q^3)$ with $|G|=q^2+q+1$. Then $G\leq\PSU(3,q^3)$.
\end{lemma}

\begin{proof}
If $\PGU(3,q^3)=\PSU(3,q^3)$ there is nothing to prove, hence we can assume that $\PSU(3,q^3)$ has index $3$ in $\PGU(3,q^3)$. Then $\gcd(3,q^3+1)=3$, or equivalently $\gcd(3,q+1)=3$, so that $3$ does not divide $q^2+q+1=|G|$. If $G\not\subseteq\PSU(3,q^3)$, then $\PGU(3,q^3)=G\cdot\PSU(3,q^3)$ and $G$ has a subgroup $G\cap\PSU(3,q^3)$ of index $[\PGU(3,q^3):\PSU(3,q^3)]=3$, contradiction.
\qed\end{proof}

\begin{proposition}\label{secondovalore}
There is no Galois-covering $\cH_{q^3}\rightarrow\cX_q$ of degree $q^2+q+1$.
\end{proposition}

\begin{proof}
Suppose such a Galois-covering exists, say $\cX_q\cong\cH_{q^3}/G$. Then $G\leq\PSU(3,q^3)$ by Lemma \ref{lem}, and we can apply Theorem \ref{MH}.

{\bf Case \textit{i)}} Let $\cH_{q^3}$ have equation $Y^{q^3+1}=X^{q^3}+X$; up to conjugation, $G$ fixes the ideal point $P_\infty$. The stabilizer $S$ of $P_\infty$ in $\PGU(3,q^3)$ is the semidirect product $P\rtimes H$, where $P$ is the unique Sylow $p$-subgroup of $S$ of order $q^9$, and $H$ is a cyclic group of order $q^6-1$ generated by $$\alpha_a:(X,Y,T)\mapsto(a^{q^3+1}X,aY,T),$$
where $a$ is a primitive $(q^6-1)$-th root of unity; $H$ fixes two $\mathbb{F}_{q^3}$-rational points of $\cH_{q^3}$ and acts semiregularly on the other points (see \cite[Section 4]{GSX}). Since $P\triangleleft S$, $|P|$ and $|H|$ are coprime, and $|G|$ divides $|H|$, then $G\subset H$, and $G=\langle\alpha_b\rangle$, where $b=a^{(q^3+1)(q-1)}$. Now consider the group $\bar G=\langle\alpha_c\rangle\subset H$, where $c=a^{q-1}$; $G$ is normal in $\bar G$ with index $q^3+1$. The automorphism group $\bar G/G$ has two fixed points on $\cH_{q^3+1}/G$ and all other orbits are long; then the number of $\mathbb{F}_{q^6}$-rational points of $\cH_{q^3}$v is congruent to $2$ modulo $q^3+1$. On the other side, the $\mathbb{F}_{q^6}$-maximal curve $\cX_q$ has genus $(q-1)(q^3-q)/2$, hence the number of $\mathbb{F}_{q^6}$-rational points of $\cX_q$ is $q^7-q^5+q^4+1\equiv q^2+1\,(\mod\,q^3+1)$, contradiction.

{\bf Case \textit{ii)}} Let $\cH_{q^3}$ have the Fermat equation $X^{q^3+1}+X^{q^3+1}+1=0$; up to conjugation, $G$ fixes the affine point $(0,0)$ and the line at infinity $\ell:T=0$. The action of $G$ on $\ell$ is faithful. In fact, if $g\in G$ fixes $\ell$ pointwise, then $g:(X,Y,T)\mapsto(X,Y,\lambda T)$ is a homology whose order divides $q^3+1$; on the other hand, the order of $g$ divides $|G|=q^2+q+1$, hence $g$ is the identity since $q$ is even. Therefore, as in the proof of Proposition \ref{punto-retta}, $G$ is isomorphic to a subgroup of $\PGL(2,q^6)$. Since $|G|$ is odd and coprime with $p$, then by Theorem \ref{Di} $G$ is cyclic. Moreover, since $|G|$ divides $q^6-1$, then $G$ has two fixed points $P_1,P_2\in\ell$ and acts semiregularly on $\ell\setminus\left\{P_1,P_2\right\}$ (see \cite[Hauptsatz 8.27]{Hu}). Since $|\ell\cap\cH_{q^3+1}|\equiv 2\,(\mod\,q^2+q+1)$, this implies $P_1,P_2\in\cH_{q^3+1}$. Therefore we can repeat the argument of Case $i)$ to get a contradiction.

{\bf Cases \textit{iii)} and \textit{iv)}} The order of $G$ does not divide the order of these maximal subgroups.

{\bf Cases \textit{v)}} $G$ acts on the $q^6+1$ $\mathbb{F}_{q^6}$-rational points of a conic $\cC$; as in Case $ii)$, $G=\langle g\rangle$ is isomorphic to a cyclic subgroup $\Gamma\leq\PGL(2,q^6)$ acting on a line $\ell$ with no short orbits but two fixed points. The action of $G$ on $\cC$ is equivalent to the action of $\Gamma$ on $\ell$, see \cite[Chapt. VIII, Thm. 15]{VY}; hence $G$ has no short orbits on $\cC$ but two fixed points $P_1,P_2$.
If $G$ has a fixed $\mathbb{F}_{q^6}$-point on $\cH_{q^3}$, then argue as in Case $i)$. Otherwise, $P_1,P_2\notin\cH_{q^3}$, and by \cite[Par. 2]{M}, \cite[Page 141]{H}, we have that $G$ fixes a third point $P_3$ such that $T=\left\{P_1,P_2,P_3\right\}$ is a self-polar triangle. Let $\cH_{q^3}$ be given in the Fermat form, then up to conjugation $T$ is the fundamental triangle and $g$ has the form $g:(X,Y,T)\mapsto(\lambda X,\mu Y,T)$. Then the order of $g$ divides $q^3+1$, contradicting $|G|=q^2+q+1$.

{\bf Cases \textit{viii)} to \textit{xii)} and Case \textit{xv)}} $|G|$ does not divide the order of these maximal subgroups.

{\bf Cases \textit{vi), vii), xiii)}, and \textit{xiv)}} Note that, if $K$ is a group containing $\PSU(3,p^h)$ as a normal subgroup of index $3$, then the orders of maximal subgroups of $K$ are three times the orders of maximal subgroups of $\PSU(3,p^h)$. With this observation, by applying Theorem \ref{MH} to $\PSU(3,p^m)$, it is easily seen that $|G|$ does not divide the orders of maximal subgroups of $\PSU(3,p^m)$ nor three times these orders.
\qed\end{proof}

\begin{lemma}\label{quantisylow}
Let $G\leq\PGU(3,q^3)$ with $|G|=q(q+1)$. Then the number of Sylow $p$-subgroups is either $1$ or $q+1$.
\end{lemma}

\begin{proof}
Let $Q_1,\ldots,Q_{n}$ be the Sylow $p$-subgroups of $G$, and let $P_i\in\cH_{q^3}$ be the unique rational point fixed by $Q_i$, $i=1,\ldots,n$. Assume $n>1$; note that $n\leq q+1$, as $Q_i\cap Q_j$ is trivial for $i\neq j$. Since $G$ has no fixed points and $Q_i$ is semiregular on $\cH_{q^3}\setminus\left\{P_i\right\}$, then the length of the orbit $\cO_{P_1}$ of $P_1$ under $G$ is at least $q+1$; on the other side, the stabilizer of $P_1$ in $G$ has length at least $q$, since it contains $Q_1$. Therefore $|\cO_{P_1}|=q+1$ by the orbit-stabilizer theorem. If $P\in\cO_{P_1}$, then the stabilizer of $P$ in $G$ has order $q$, hence $P=P_i$ for some $i\in\left\{2,\ldots,n\right\}$. Therefore $\cO_{P_1}=\left\{P_1,\ldots,P_n\right\}$ and the thesis follows.
\qed\end{proof}

\begin{proposition}\label{1sylow}
Let $G\leq\PGU(3,q^3)$ with $|G|=q(q+1)$. If $G$ has a unique Sylow $p$-subgroup $Q$, then $\cX_q\not\cong\cH_{q^3}/G$.
\end{proposition}

\begin{proof}
Let $\cH_{q^3}$ be given in the norm-trace form $Y^{q^3+1}=X^{q^3}+X$. Since $Q$ is normal in $G$, then $G$ fixes the unique fixed point of $Q$ on $\cH_{q^3}$; up to conjugation, this is the ideal point $P_\infty$. By Hall's theorem, we can assume that $G=Q\rtimes\langle\alpha_\lambda\rangle$, where $$\alpha_\lambda:(X,Y,T)\mapsto(\lambda^{q^3+1}X,\lambda Y,T)=(X,\lambda Y,T),$$ with $\lambda$ primitive $(q+1)$-th root of unity. Suppose $\cX_q\cong\cH_{q^3}/G$, in particular the genus of $\cX_q$ equals the genus of $\cH_{q^3}/G$, which is given in \cite[Thm. 4.4]{GSX}. With the notations of \cite[Section 4]{GSX}, this implies $v=0$ and $q=p^w$, that is, the elements of $Q$ have the form
$$\beta_c:(X,Y,T)\mapsto(X+cT,Y,T),$$
with $c^{q^3}+c=0$. The set $\left\{c\in\mathbb{F}_{q^6}\mid\beta_c\in Q\right\}$ is an additive group, isomorphic to $Q$. Then $Q\cong\left\{c\in\mathbb{F}_{q^6}\mid L(c)=0\right\}$, where $L\in\mathbb{F}_{q^6}[X]$ is a linearized polynomial of degree $q$ dividing $X^{q^3}+X$, and there is a linearized polynomial $F\in\mathbb{F}_{q^6}[X]$ of degree $q^2$ such that $F(L(X))=X^{q^3}+X$ (see \cite[Theorems 3.62, 3.65]{LN}).
Therefore the quotient curve $\cH_{q^3}/G$ has equation $Y^{q^2-q+1}=F(X)$.

The thesis follows, if we show that there cannot exist an $\mathbb F_{q^6}$-isomorphism $\varphi:\cC\rightarrow\cX_{q}$, where $\cX_q: V^{q^2-q+1}=U^{q^2}-U$ and $\cC$ is a curve with equation $Y^{q^2-q+1}=F(X)$, with $F\in\mathbb{F}_{q^6}[X]$ a linearized divisor of $X^{q^3}+X$ of degree $q^2$.

Suppose such a $\varphi$ exists. By \cite[Thm. 12.11]{HKT}, the ideal points $P_\infty\in\cX_q$, $Q_\infty\in\cC$ are the unique fixed points of the respective automorphism groups $\aut(\cX_q)$, $\aut(\cC)$, hence $\varphi(Q_\infty)=P_\infty$. Moreover, the coordinate functions have pole divisors
$$ div(u)_\infty=(q^2-q+1)P_\infty,\; div(v)_\infty=q^2P_\infty,\quad div(x)_\infty=(q^2-q+1)Q_\infty,\; div(y)_\infty=q^2Q_\infty, $$
and the Weierstrass semigroups at the ideal points are $H(P_\infty)=H(Q_\infty)=\langle q^2-q+1,q^2\rangle$ (see \cite[Lemmata 12.1, 12.2]{HKT}). By Riemann-Roch theory (see \cite[Chapt. 1]{Sti}), it is easily seen that $\left\{1,x\right\}$ is a basis of the Riemann-Roch space $\mathcal{L}((q^2-q+1)P_\infty)$ associated to $(q^2-q+1)P_\infty$, and $\left\{1,x,y\right\}$ is a basis of $\mathcal{L}(q^2P_\infty)$. Then there exist $a,b,c,d,e\in\mathbb{F}_{q^6}$, $a,d\neq0$, such that $\varphi^*(u)=ax+b$ and $\varphi^*(v)=cx+dy+e$, where $\varphi^*:\mathbb{F}_{q^6}(\cX_q)\rightarrow\mathbb{F}_{q^6}(\cC)$ is the pull-back of $\varphi$; equivalently, $\varphi(X,Y,T)=(aX+b,cX+dY+e,T)$. Then the polynomial identity
$$ \left(aX+b\right)^{q^2}-\left(aX+b\right)-\left(cX+dY+e\right)^{q^2-q+1}=k\left(F(X)-Y^{q^2-q+1}\right) $$
holds, for some $k\in\bar{\mathbb{F}}_{q^6}$, $k\neq0$. By direct calculation and comparison of  the coefficients, we get the constraints $c=e=0$, $b\in\mathbb{F}_{q^2}$, $k=d^{q^2-q+1}$, which imply
$$ F(X)=k^{-1}a^{q^2}X^{q^2}-k^{-1}aX. $$
It is easily checked that the conventional $p$-associate of the linearized polynomial $F(X)$ is not a divisor of the conventional $p$-associate of $X^{q^3}+X$, hence $F(X)$ is not a divisor of $X^{q^3}+X$.
\qed\end{proof}

\begin{lemma}\label{tantisylow}
Let $G\leq\PGU(3,q^3)$ with $|G|=q(q+1)$. If $G$ has $q+1$ distinct Sylow $p$-subgroup $Q_1,\ldots Q_{q+1}$, then $G\cong(\mathbb{Z}_{p'})^s\rtimes Q_1$, where $p'$ is a prime and $(p')^s=q+1$.
\end{lemma}

\begin{proof}
By Lemma \ref{quantisylow}, the points $P_1,\ldots,P_{q+1}$ fixed respectively by $Q_1,\ldots,Q_{q+1}$ constitute a single orbit $\cO$ under the action of $G$. By Burnside's Lemma, $G$ is sharply $2$-transitive on $\cO$. Then, by \cite[Thm. 20.7.1]{Hall}, $G$ is isomorphic to the group of affine trasformations of a near-field $F$; moreover, $G$ has a regular normal subgroup $N$, hence $G=N\rtimes Q_1$. The order $f$ of $F$ satisfies $q(q+1)=(f-1)f$, which implies $f=q+1$. By this condition, $F$ cannot be one of the seven exceptional near-fields listed in \cite{Z}, hence $F$ is a Dickson near-field, see \cite[Thm. 20.7.2]{Hall} for a description. In particular, $N$ is isomorphic to the additive group $(\mathbb{Z}_{p'})^s$ of a finite field $\mathbb{F}_{(p')^s}$.
\qed\end{proof}

\begin{proposition}\label{q+1sylow}
Let $G\leq\PGU(3,q^3)$ with $|G|=q(q+1)$. If $G$ has $q+1$ distinct Sylow $p$-subgroup $Q_1,\ldots Q_{q+1}$, then $\cX_q\not\cong\cH_{q^3}/G$.
\end{proposition}

\begin{proof}
We use the notations of Lemma \ref{tantisylow} and assume $\cX_q\cong\cH_{q^3}/G$.

Suppose $q$ is odd. Then all involutions of $\PGU(3,q^3)$ are conjugate, and they are homologies of $\PG(2,q^6)$, see \cite[Lemma 2.2]{KOS}. Two homologies commute if and only if the center of each lies on the axis of the other (see for example \cite[Thm. 5.32]{Cox}), hence the maximum number of involutions commuting pairwise is $3$, since their centers are three non-collinear points. Then $(p')^s=4$ and $q=3$, against the assumptions of this section.

Suppose $q$ is even. $Q_1$ is isomorphic to the multiplicative group of $F$, hence it is a metacyclic group, see for example \cite[Ex. 1.19]{Cam}; moreover, $Q_1$ has exponent $2$ or $4$ by \cite[Lemma 2.1]{KOS}. Therefore $q\in\left\{2,4,8,16\right\}$.
The case $q=2$ is excluded. If $q=16$, then $F$ is a Dickson near-field of prime order $17$, hence $F$ is a field, against the exponent $2$ or $4$ of $Q_1$. Then $q=4$ or $q=8$.

We use the Riemann-Hurwitz formula \cite[Thm. 3.4.13]{Sti} on the covering $\cH_{q^3}\rightarrow\cX_q\cong\cH_{q^3}/G$, in order to get a contradiction on the degree $\Delta=\left(2g(\cH_{q^3})-2\right)-|G|\left(2g(\cX_q)-2\right)$ of the Different. By \cite[Thm. 3.8.7]{Sti} 
$$ \Delta=\sum_{\sigma\in G\setminus\left\{id\right\}}i(\sigma), $$
%$$\Delta=\sum_{P\in\cH_{q^3}}\sum_{i=0}^{\infty}\left(\left|G_P^{(i)}\right|-1\right), $$
where the contributions $i(\sigma)\geq0$ to $\Delta$ satisfy the following:
\begin{itemize}
\item If $\sigma$ has order $2$, then $i(\sigma)=q^3+2$; if $\sigma$ has order $4$, then $i(\sigma)=2$ (see \cite[Eq. (2.12)]{Sti}).
\item If $\sigma$ is odd, then $i(\sigma)$ equals the number of fixed points of $\sigma$ on $\cH_{q^3}$, see \cite[Cor. 3.5.5]{Sti}; moreover, by \cite[pp. 141-142]{H}, either $\sigma$ has exactly $3$ fixed points or $\sigma$ is a homology. In the former case $i(\sigma)\leq3$, in the latter $i(\sigma)=q^3+1$.
\end{itemize}

Let $q=4$, hence $\Delta=470$ and $G=\mathbb{Z}_5\rtimes Q_1$. If $Q_1\cong\mathbb{Z}_2\times\mathbb{Z}_2$, then $G$ has $15$ involutions, whose contributions to $\Delta$ sum up to $990>\Delta$. Then $Q_1\cong\mathbb{Z}_4$, and the contributions to $\Delta$ of the $Q_i$'s sum up to $5\cdot66+10\cdot2=350$. The non-trivial elements of $\mathbb{Z}_5$ are generators of $\mathbb{Z}_5$, then either all of them are homologies or all of them fix $3$ points. In the former case their contribution to $\Delta$ exceeds $120$, in the latter their contribution is smaller than $120$.

Let $q=8$, hence $\Delta=7758$ and $G=(\mathbb{Z}_3\times\mathbb{Z}_3)\rtimes Q_1$. If $Q_1$ has more than one involution, then the involutions of $G$ contribute to $\Delta$ of at least $18\cdot514>\Delta$. Then $Q_1$ is the quaternion group, and the $Q_i$'s contribute to $\Delta$ of $9\cdot514+54\cdot2=4734$. The contribution to $\Delta$ of any non-trivial element of $\mathbb{Z}_3\times\mathbb{Z}_3$ is either $513$ or less than $4$, hence they cannot sum up to $\Delta-4734$.
\qed\end{proof}
\vspace*{.2cm}
By Lemma \ref{quantisylow} and Propositions \ref{1sylow}, \ref{q+1sylow}, we have shown the following result.

\begin{proposition}\label{terzovalore}
There is no Galois-covering $\cH_{q^3}\rightarrow\cX_q$ of degree $q^2+q$.
\end{proposition}
\vspace*{.2cm}
Finally, Theorem \ref{result2} follows from Propositions \ref{possibilivalori}, \ref{primovalore}, \ref{secondovalore}, and \ref{terzovalore}.


\begin{thebibliography}{99}

\bibitem{ABQ} Abd\'on, M., Bezerra, J., Quoos, L.: Further examples of maximal curves. J. Pure Appl. Algebra {\bf 213} (6), 1192--1196 (2009)

\bibitem{Cam} Cameron, P.J.: Permutation Groups. Cambridge University Press (1999)

\bibitem{CKT} Cossidente, A., Korchm\'aros, G., Torres, F.: On curves covered by the Hermitian curve. J. Algebra {\bf 216} (1), 56--76 (1999)

\bibitem{CKT2} Cossidente, A., Korchm\'aros, G., Torres, F.: Curves of large genus covered by the Hermitian curve. Comm. Algebra {\bf 28} (10), 4707--4728 (2000)

\bibitem{Cox} Coxeter, H.S.M.: The Real Projective Plane, 3rd edn. Springer, New York (1993)

\bibitem{D} Dickson, L. E.: Linear groups with an exposition of the Galois field theory. Teubner, Leipzig (1902)

\bibitem{DM} Duursma, I., Mak, K.-H.: On maximal curves which are not Galois subcovers of the Hermitian curve. Bull. Braz. Math. Soc. (N.S.) {\bf 43} (3), 453--465 (2012)

\bibitem{FT} Fuhrmann, R., Torres, F.: On Weierstrass points and optimal curves. Rend. Circ. Mat. Palermo Suppl. 51 (Recent Progress in Geometry, Ballico E, Korchm\'aros G, (Eds.)), 25--46 (1998)

\bibitem{G} Garcia, A.: Curves over finite fields attaining the Hasse-Weil upper bound. In: European Congress of Mathematics, vol. II (Barcellona 2000), Progr. Math. 202, pp. 199--205. Birkh\"auser, Basel (2001)

\bibitem{G2} Garcia, A.: On curves with many rational points over finite fields. In: Finite Fields with Applications to Coding Theory, Cryptography and Related Areas, pp. 152--163. Springer, Berlin (2002)

\bibitem{GGS} Garcia, A., G\"uneri, C., Stichtenoth, H.: A generalization of the Giulietti-Korchm\'aros maximal curve. Advances in Geometry {\bf 10} (3), 427--434 (2010)

\bibitem{GS} Garcia, A., Stichtenoth, H.: Algebraic function fields over finite fields with many rational places. IEEE Trans. Inform. Theory {\bf 41}, 1548--1563 (1995)

\bibitem{GS2} Garcia, A., Stichtenoth, H.: A maximal curve which is not a Galois subcover of the Hermitian curve. Bull. Braz. Math. Soc. (N.S.) {\bf 37} (1), 139--152 (2006)

\bibitem{GSX} Garcia, A., Stichtenoth, H., Xing, C.P.: On Subfields of the Hermitian Function Field. Compositio Math. {\bf 120}, 137--170 (2000)

\bibitem{GT} Garcia, A., Torres, F.: On unramified coverings of maximal curves. In: Proceedings of AGCT-10 2005, S\'emin. Congr. 21, 35--42 (2010)

\bibitem{GK} Giulietti, M., Korchm\'aros, G.: A new family of maximal curves over a finite field. Math. Ann. {\bf 343} (1), 229--245 (2009)

\bibitem{GOS} G\"uneri, C., \"Ozdemir, M., Stichtenoth, H.: The automorphism group of the generalized Giulietti-Korchm\'aros function field. Advances in Geometry {\bf 13} (2), 369--380 (2013)

\bibitem{GMP} Guralnick, R., Malmskog, B., Pries, R.: The automorphism of a family of maximal curves. J. Algebra {\bf 361}, 92--106 (2012)

\bibitem{Hall} Hall, M.: The Theory of Groups. Macmillan, New York (1959)

\bibitem{H} Hartley, R. W.: Determination of the ternary collineation group whose coefficients lie in the $GF(2^n)$. Ann. of Math. Second Series {\bf 27} (2), 140--158 (1925)

\bibitem{HKT} Hirschfeld, J.W.P., Korchm\'aros, G., Torres, F.: Algebraic Curves over a Finite Field. Princeton Series in Applied Mathematics, Princeton (2008)

\bibitem{HP} Hughes, D.R., Piper, F.C.: Projective planes. Graduate Text in Mathematics {\bf 6}. Springer, New York (1973) 

\bibitem{Hu} Huppert, B.: Endliche Gruppen I. Die Grundlehren der Mathematischen Wissenschaften {\bf 134}. Springer, Berlin (1967)

\bibitem{KOS} Kantor, W.N., O'Nan M.E., Seitz, G.M.: $2$-Transitive Groups in Which the Stabilizer of Two Points is Cyclic. J. Algebra {\bf 21}, 17--50 (1972)

\bibitem{L} Lachaud, G.: Sommes d'Eisenstein et nombre de points de certaines courbes alg\'ebriques sur les corps finis. C.R. Acad. Sci. Paris 305, S\'erie I, 729--732 (1987)

\bibitem{LN} Lidl, R., Niederreiter, H.: Introduction to finite fields and their applications. Cambridge University Press, Cambridge (1986)

\bibitem{Mac} Macdonald, I.D.: The theory of groups. Oxford University Press, Oxford (1968)

\bibitem{Mak} Mak, K.-H.: On congruence function fields with many rational points. PhD Thesis. Available at www.ideals.illinois.edu

\bibitem{M} Mitchell, H.H.: Determination of the ordinary and modular ternary linear groups. Trans. Amer. Math. Soc. {\bf 12} (2), 207--242 (1911)

\bibitem{Sti} Stichtenoth, H.: Algebraic function fields and codes, 2nd edn. Graduate Texts in Mathematics {\bf 254}. Springer, Berlin (2009)

\bibitem{V} van der Geer, G.: Curves over finite fields and codes. In: European Congress of Mathematics, vol. II (Barcellona 2000), Progr. Math. 202, pp. 225--238. Birkh\"auser, Basel (2001)

\bibitem{V2} van der Geer, G.: Coding theory and algebraic curves over finite fields: a survey and questions. In: Applications of Algebraic Geometry to Coding Theory, Physics and Computation, NATO Sci. Ser. II Math. Phys. Chem. 36, pp. 139--159. Kluwer, Dordrecht (2001)

\bibitem{VY} Veblen, O., Young, J.W.: Projective Geometry. The Atheneum Press, Boston (1910)

\bibitem{Z} Zassenhaus, H.: \"Uber endliche Fastk\"orper. Abh. Math. Sem. Univ. Hamburg {\bf 11}, 132--145 (1936)

\end{thebibliography}
\end{document}